\documentclass{amsart}
\usepackage{algorithm}

\usepackage{algpseudocode}
\RequirePackage[OT1]{fontenc}
\RequirePackage{amsthm,amsmath, amssymb, mathrsfs, graphicx, listings, natbib}
\usepackage{marginnote}
\usepackage{todonotes}
\usepackage{subfigure}
\usepackage{graphicx}
\usepackage{caption}
\usepackage{multirow}

\usepackage{url} 
\usepackage{hyperref} 

\usepackage{algpseudocode}
\makeatletter
\let\OldStatex\Statex
\renewcommand{\Statex}[1][3]{%
	\setlength\@tempdima{\algorithmicindent}%
	\OldStatex\hskip\dimexpr#1\@tempdima\relax}
\makeatother

\usepackage{changebar} 
  \usepackage{bm}

\usepackage
{todonotes}

\usepackage{tikz}
\usetikzlibrary{arrows, calc, shapes}
\usetikzlibrary{decorations.pathreplacing}
\usetikzlibrary{positioning}
\tikzstyle{ov}=[shape=rectangle,
                draw=black!80,
                minimum height=0.6cm,
                minimum width=.9cm,
                thick]

\tikzstyle{lv}=[shape=circle,draw=black!80,thick,minimum width=1.3cm]

\tikzstyle{err}=[circle,
     minimum size=1mm,draw,thick]

\tikzset{
>=stealth',
help lines/.style={dashed, thick},
axis/.style={<->},
important line/.style={thick},
connection/.style={thick, dotted},
}

\newtheoremstyle{theorem}
{10pt} 
{10pt} 
{\sl} 
{\parindent} 
{\bf} 
{. } 
{ } 
{} 
\theoremstyle{theorem}
\newtheorem{theorem}{Theorem}
\newtheorem{corollary}{Corollary}

\DeclareMathOperator*{\argmin}{argmin}

\def\d{{\rm d}}

\newcommand{\half} { \frac{1}{2} }

\newcommand{\inv} { {-1} }

\newcommand{\mlt}{T_{\text{ML}}}
\newcommand{\sba}{T_{\text{SB}}}
\newcommand{\sbb}{T_{\text{SS}}}

\newcommand{\npconv} { \xrightarrow[n \rightarrow \infty]P }

\newcommand{\ndconv} { \xrightarrow[n \rightarrow \infty]{D} }

\title[New testing procedures for moment structure models]{New testing procedures \\ for structural equation modeling}

\author{Steffen Gr\o nneberg}
\address{Department of Economics\\
BI Norwegian Business School\\
Oslo, Norway 0484}
\email{steffeng@gmail.com}

\author{Nj\aa l Foldnes}
\address{Department of Economics\\
	BI Norwegian Business School\\
	Stavanger, Norway 4014}
\email{njal.foldnes@bi.no}

\begin{document}

\begin{abstract}
	We introduce and evaluate a new class of hypothesis testing procedures for moment structures. The methods are valid under weak assumptions and includes the 
 well-known Satorra-Bentler adjustment as a special case. The proposed procedures applies also to difference testing among nested models.  
 We prove the consistency of our approach. We introduce a bootstrap selection mechanism to optimally choose a p-value approximation for a given sample.
	Also, we propose bootstrap procedures for assessing the asymptotic robustness (AR) of the normal-theory maximum likelihood test, and for the key assumption underlying the Satorra-Bentler adjustment (Satorra-Bentler consistency).
	Simulation studies indicate that our new p-value approximations performs well even under severe nonnormality and realistic sample sizes, but that our tests for AR and Satorra-Bentler consistency require very large sample sizes to work well.
	R code for implementing our methods is provided.
\end{abstract}

\date{\today}

\maketitle

\section{Introduction }

In testing hypotheses in psychometrics, test statistics often converge in law  to a mixture of independent chi squares, under the null hypothesis of correct model specification. This paper presents novel methods for calculating p-values based on such test statistics, and a novel selection procedure aimed at identifying the best p-value among any given set of candidate p-value procedures. 
Although the proposed methods can be used in the general setting of moment structure inference, 
we here focus on the framework of structural equation modeling (SEM).

As shown in \citet{shapiro1983asymptotic} and \citet{satorra1989alternative},
 a large class of p-values in the context of SEM originates from convergence in distribution results (derived under the null hypothesis) for a test statistic $T_n$ based on $n$ observations, of the form
\begin{equation} \label{mixture}
T_n \ndconv \sum_{j=1}^d \lambda_j Z_j^2, \qquad Z_1, \ldots, Z_d \sim N(0,1) \text{ IID},
\end{equation}
where $\lambda = (\lambda_1, \ldots, \lambda_d)'$ consists of unknown population parameters. 
If $\lambda$ was known, eq.~\eqref{mixture} motivates the ``oracle'' p-value
\begin{equation} \label{equ::oracledef}
p_{n} =  P \left( \sum_{j=1}^d \lambda_j Z_j^2 > T_n \right).
\end{equation}
The above probability is with respect to $Z_1, \ldots, Z_d$, while $T_n$ is considered fixed.
In a practical setting $\lambda$ is however unknown. 
Let   $\hat \lambda$ be a a
consistent estimator   of $\lambda$, i.e., $\hat \lambda \npconv \lambda$. 
In the present article we propose to estimate  $p_{n}$ by
\begin{equation} \label{equ::phatdef}
\hat p_{n} =P \left( \sum_{j=1}^d \hat \lambda_j Z_j^2 > T_n \right),
\end{equation}
where the probability is with respect to $Z_1, \ldots, Z_d$.

 We show that for large samples, the error originating from replacing $\lambda$ with $\hat \lambda$ is vanishing, so that $\hat p_n - p_n$ converges  in probability to $0$. The estimator $\hat p_n$ defined above is the canonical member of a new class of estimators, obtained by grouping the $\hat \lambda_j$ by magnitude and replacing them by group means in order to reduce variance in $\hat p_n$.
Although the idea behind this new class of p-value approximations is simple,  we are unaware that it is found previously in the literature. 

Since we introduce a whole class of p-value approximations, where no member seems to be uniformly best in all conditions, we also introduce a selector to aid the user in choosing which p-value approximation to apply. The core idea of this selector is to choose the p-value approximation whose distribution is closest to the uniform, as measured by the supremum distance. This is achieved through the non-parametric bootstrap and is seen to work very well in our simulation experiments.

The paper is structured as follows. 
In Section \ref{section::theory} we review fit statistics of moment structures with a special emphasis on 
the well-known Satorra-Bentler (SB) statistic.
Section \ref{section::testprocedure} proposes 
  a class of new procedures, that incorporates the SB statistic as a special case, to evaluate model fit and parameter restrictions in covariance models. We give conditions under which the estimators are consistent, which implies the fundamental p-value property of converging in distribution to a uniform distribution.
  In Section \ref{section::selection} we  introduce a bootstrap procedure that selects, for a given sample, a good candidate among a list of p-value approximations.
  Next, in Section \ref{section::ar} we introduce a bootstrap test for assessing whether
  the normal-theory maximum likelihood  test statistic may be trusted, i.e., whether asymptotic robustness holds. Also, we introduce a test for the consistency of the SB statistic, which may help decide whether the SB statistic may be trusted. 
Monte Carlo results on the performance of the proposed new methods are presented in  Section \ref{section::montecarlo}. In the final section we discuss our findings and point out further directions for research. 
Proofs of theoretical results are presented in the appendix.

\section{Fit statistics for moment structure models} \label{section::theory}

Consider a  $p$-dimensional  vector of population moments  $ \sigma^\circ$. In covariance modeling, $\sigma^\circ$ consists of second-order moments, but in more general structural equation models the means may also be included in $ \sigma^\circ$.  
The corresponding sample moment vector $s$ is assumed to converge in probability to $\sigma^\circ$, i.e., $s \npconv \sigma^\circ$, and be asymptotically normal, i.e.̃, $\sqrt{n} ( s - \sigma^\circ) \ndconv N(0, \Gamma)$.
A structural equation model  implies a certain parametrization $\theta \mapsto \sigma(\theta)$ with $\theta$ varying in a set $\Theta$.
Let the free parameters in the proposed  model be contained  in the $q$-vector $\theta$.
 The model has degrees of freedom given by $d=p-q$.

The model is said to be correctly specified if there is a $\theta^\circ \in \Theta$ such that $\sigma(\theta^\circ) = \sigma^\circ$.
A very general class of estimators for 
$\theta^\circ$ introduced by \cite{browne1982covariance, browne1984asymptotically}  is obtained by minimising   
  discrepancy functions $F=F(s, \sigma)$ 
  that obey the following three conditions: $F(s, \sigma) \geq 0$ for all $s, \sigma$; $F(s, \sigma)=0$ if and only if $s = \sigma$; and $F$ is twice continuously differentiable jointly. That is, we consider estimators obtained as
\[
\hat \theta = \argmin_{\theta \in \Theta} F(s, \sigma(\theta)).
\]
It is well known that the widely used normal-theory maximum likelihood (NTML)
estimator is such a minimal discrepancy estimator.

Similarly, we may define the least false parameter configuration, which we  denote with $\theta^\circ$. That is,
\[
\theta^\circ = \argmin_{\theta \in \Theta} F(\sigma^\circ, \sigma(\theta)).
\]
Irrespective of the correctness of the model, we have $\hat \theta \npconv \theta^\circ$ under mild regularity conditions. 

One, out of several mainly asymptotically equivalent \citep[see][]{satorra1989alternative} ways of assessing the correctness of the model is to study $T_n =  n F(s, \sigma(\hat \theta))$.
If the model is misspecified, i.e. if $\sigma^\circ \neq \sigma(\theta^\circ)$, then  $T_n \rightarrow \infty$ since $s \npconv \sigma^\circ \neq \sigma(\theta^\circ)$. 
Under correct model specification and other assumptions presented in \citet{shapiro1983asymptotic} and \citet{satorra1989alternative}, we have
$
T_n = \sqrt{n} (s - \sigma^\circ)' U \sqrt{n} (s-\sigma^\circ) + o_P(1).
$
Assuming (for simplicity) that $\Delta' V \Delta$ is non-singular (see comment immediately following eq.~(9) in \citet{satorra1989alternative}) where  $\Delta$ is the $p \times q$ derivative matrix $ \partial  \sigma (\theta) / \partial \theta'$ evaluated at $\theta^\circ$, and $V = \frac{1}{2} \frac{\partial^2 F(s, \sigma) }{\partial s \partial \sigma}$, evaluated at $(\sigma^\circ, \sigma^\circ)$, we have 
\begin{equation} \label{equ::Udef}
U={V}- {V} \Delta \left\{\Delta' {V} \Delta  \right\} ^{-1}\Delta'{V}.
\end{equation}
Note that $U$ has rank $d$.
Since we assume $\sqrt{n}(s - \sigma^\circ) \ndconv Q \sim N(0,\Gamma)$, the continuous mapping theorem now implies that 
$
T_n \ndconv Q' U Q.
$
By Theorem 1 in \citet{Box}, we have 
\begin{equation} \label{equ::box}
T_n  \ndconv Q' U Q = \sum_{j=1}^d \lambda_j Z_j^2, \qquad Z_1, \ldots, Z_d \sim N(0,1) \text{ IID},
\end{equation}
where $\lambda_1, \ldots, \lambda_d$ are the $d$ non-zero eigenvalues of $U \Gamma$  under the standard  scaling of the eigenvectors. That is, the parameters $\lambda$ in eq.~\eqref{mixture} are eigenvalues of a certain matrix that depends both on the underlying distribution and on the proposed model. Note that estimating $U$ and $\Gamma$ is a standard problem in moment models which we will not discuss in technical detail. The usual estimators are based on replacing expectations with averages of the observed data, and the true least-false parameter $\theta^\circ$ by the estimator $\hat \theta$. This is the estimator readily available in software packages such as the \textsf{R} package lavaan \citep{rosseel2012lavaan}. We here assume that consistent estimators $\hat U$,  $\hat \Gamma$  and $\hat \lambda$ are given. We may use the plug-in method to form $\hat \lambda$, so that $\hat \lambda$ is the $d$ non-zero eigenvalues of $\hat U \hat \Gamma$ under the standard  scaling of the eigenvectors.

Note that the asymptotically distribution-free (ADF) estimator of \citet{browne1984asymptotically}, where
the estimate is obtained by minimising a quadratic form whose weight matrix is the inverse of a distribution-free
estimate of $\Gamma$, yields a test statistic $T_{\text{ADF}}$ whose population eigenvalues are all equal to one. 
Hence ADF estimation leads to 
consistent p-values for model fit. However, ADF estimation is unstable in small samples, and it is well-known that $T_{\text{ADF}}$
has unacceptably poor performance in small and medium samples sizes \citep{Curran, hu92}.
Another test statistic with consistent p-value approximation is the residual-based test statistic \cite[eq. 2.20]{browne1984asymptotically}, which is not of
 the form $T_n = n F(s, \sigma(\hat \theta))$ investigated in the present article.
 Unfortunately this statistic suffers from
the same lack of acceptable finite-sample performance as $T_{\text{ADF}}$. 
Therefore, a more popular approach has been to  use normal-theory based estimators, and to correct the test statistic for
non-normality in the data. We now proceed to describe such methods.

Based on the convergence result in eq.~\eqref{mixture}, \cite{SatorraBentler94} proposed to rescale $T_n$ by dividing it by the mean value of the eigenvalues to form 
\[
\sba= \frac{ T_n}{\hat c},
\]
where $\hat c=\frac{\sum_{j=1}^{d}\hat \lambda_j}{d}$.
Using $\sba$ as a test statistic is a widely used SEM practice under conditions of non-normal data.  Simulation studies report that $\sba$ outperforms the NTML fit statistic $\mlt$ in such conditions,
but that Type I error rates under $\sba$ 
are seriously inflated under substantial excess kurtosis in the data \citep{bentler1999structural, nevitt2004evaluating, Savalei:2010, foldnes2015correcting}. Also, \cite{bentleryuan2010} theoretically demonstrated that $\sba$ departs from a chi-square with increasing dispersion of the  eigenvalues in \eqref{mixture}.

Recently \cite{asparouhov2010simple} proposed a test statistic that 
agrees with the reference chi-square distribution in both asymptotic mean and variance, obtained from $\mlt$ by scaling and shifting. This statistic,  found to perform slightly better \citep{foldnes2015correcting} than a Sattertwaithe type test statistic proposed by \cite{SatorraBentler94},  
is given by
\[
\sbb	=\sqrt{\frac{d}{tr\left( (\hat U \hat \Gamma)^2\right)}} \cdot T_n+
d-\sqrt{\frac{d  \left( tr( \hat{U} \hat{\Gamma}) \right)^2} {tr\left( (\hat U \hat \Gamma)^2\right)}}.
\]

A quite different testing methodology is offered by the so-called Bollen-Stine bootstrap \citep{bollen1992bootstrapping}, 
which is based on the non-parametric bootstrap \citep{efron1994introduction}.
Instead of starting with the fundamental result in eq.~\eqref{mixture}, one starts by  transforming the sample observations $X_i$ into $\tilde X_i = \Sigma(\hat \theta)^{1/2} S_n^{-1/2} X_i$ for $i = 1, 2, \ldots, n$, 
where $S_n$ and $\Sigma(\hat \theta)$ are the sample and model-implied covariance matrices, respectively. Noting that the model holds exactly in this transformed sample, we  proceed by assuming that the transformed sample may serve as a proxy for the population from which the original sample was drawn. 
The Bollen-Stine $p$-value is now obtained
by drawing bootstrap samples from the transformed sample, 
and calculating the 
proportion of bootstrap test statistics that exceed the test statistic obtained from the original sample. The validity of this approach is derived in \citet{beran1985bootstrap}. \cite{Nevitt:2001gg} report that Bollen-Stine bootstrapping outperformed the SB scaling approach under correct model specification at realistic sample sizes, having  Type I errors slightly below the nominal level. Despite the promising performance of the Bollen-Stine bootstrap, it seems to be relatively understudied. 
In fact,  we are not aware of any later simulation study that systematically evaluates its performance relative to other robust test statistics.

Our upcoming selection methodology to be described in Section \ref{section::selection} will fuse the two ideas discussed above,   trying to combine the strength of the convergence in eq.~\eqref{mixture} with the power of the non-parametric bootstrap. Our tests for AR and Satorra-Bentler consistency described in Section \ref{section::ar}, are also based on the non-parametric bootstrap. Before describing these bootstrap based methods, we return to the fundamental convergence result in eq.~\eqref{mixture} and present new approximations for the oracle p-value.

\section{A new class of p-value approximations} \label{section::testprocedure}

In this section we introduce and establish the consistency of a new computational technique for p-values. 
The proposed methodology applies as long as the null distribution of a test statistic is a weighted sum of independent chi squares and the weights can be estimated consistently.  
This means that the method may be used both for conventional goodness-of-fit testing of a single proposed model, and for 
 nested model comparison tests. 
Consistency   is established in Theorem \ref{theorem::consistency}, and the proof is found in Appendix \ref{section::proof}.

The convergence result in eq. \eqref{mixture} is only valid if the model is correctly specified. But we here note that $U \Gamma$ is defined also when the model is misspecified, and that the number of non-zero eigenvalues is known to be $d$ from the model configuration. We may therefore speak of and estimate $\lambda = (\lambda_1, \lambda_2, \ldots, \lambda_d)'$  without knowing if the model is correctly specified. We refer to the  p-value in \eqref{equ::phatdef} as the full p-value approximation.
We will also shortly introduce other estimators by combining the $\hat \lambda_j$ 
in eq.\eqref{equ::phatdef} in ways that may reduce variability in $\hat p_n$, although at the expense of consistency. 
This may be reasonable in situations where the full estimates are unstable, e.g,  under small sample sizes and highly non-normal data. In fact,  the familiar $\sba$ procedure may  conceptualized as an (inconsistent) p-value approximation where the $\lambda_j$ are replaced by 
the mean value of the canonical estimates,  i.e., $\hat \lambda_j^{SB} = \frac{\sum_{j=1}^{d}\hat \lambda_j}{d}$, $j=1, \ldots, d$, 
and clearly
\[
\hat p_{SB} = P \left( \sum_{j=1}^{d} \hat \lambda_j^{SB} Z_j^2 > T_n\right).
\]

We obtain a valid approximation as long as $\hat \lambda \npconv \lambda$, as the following theorem shows. Note that we make \emph{no} assumptions on $T_n$. That is, the approximation holds irrespective of the correctness of the model. Note also that we typically have $\|\hat \lambda - \lambda\| = O_P(n^{-1/2})$, i.e.,~$\sqrt{n} [\hat p_{n} - p_{n}]$ stays bounded in probability.

\begin{theorem} \label{theorem::consistency}
	Let $(T_n)$ be a sequence of random variables, and let $p_{n} = 1-  H(T_n ; \lambda)$ and $\hat p_{n} = 1 - H(T_n ; \hat \lambda)$ where $H(q ; \lambda_1, \ldots, \lambda_r) = P ( \sum_{j=1}^{d} \lambda_j Z_j^2 \leq q)$.
	If $\hat \lambda \npconv \lambda$ where $\lambda$ only has positive elements, then $\hat p_{n} - p_{n} = \| \hat \lambda - \lambda \| O_P(1)$, and hence, 
	$
	\hat p_{n} - p_{n} \npconv 0.
	$
\end{theorem}
\begin{proof}
	See Appendix \ref{section::proof}.
\end{proof}

We see that the $\sba$ procedure is a valid large sample approximation to $p_n$ if $\lambda = c \cdot (1, 1, \ldots, 1)$. If this is true in the population, Theorem \ref{theorem::consistency} implies the consistency of the $\sba$ procedure. 
The only crucial assumption of the theorem is that each $\lambda_j > 0$. Recall that in goodness of fit testing in SEM, we are guaranteed $d$ non-zero eigenvectors by the Box Theorem, see the discussion near eq.~\eqref{equ::box}. Hence this assumption is innocuous. 

A direct consequence of Theorem \ref{theorem::consistency} is that $\hat p_n$ fulfills the following property considered fundamental to p-values.
\begin{corollary}
Suppose the conditions of Theorem \ref{theorem::consistency} holds. 
If $T_n \ndconv \sum_{j=1}^d \lambda_j Z_j^2$ for $Z_1, \ldots, Z_d \sim N(0,1)$ IID, then
$
\hat p_n \ndconv U[0,1].
$
\end{corollary}
\begin{proof}
	Since $\eqref{mixture}$ holds, it follows that $p_n \ndconv U[0,1]$. Then the corollary follows  from the standard asymptotic result that if $X_n - Y_n = o_P(1)$ and $X_n \ndconv Z$ then also $Y_n \ndconv Z$.
\end{proof}

From our perspective of aiming at consistent p-values, the $\sba$ procedure is well motivated under an equality constraint among all eigenvalues. 
But if the eigenvalues differ considerably in the population, this restriction may lead to poor estimates due to a high bias.
In contrast, $\hat p_n$ is always a valid approximation for $p_n$ in that it is consistent -- and hence asymptotically unbiased. However, in finite samples the variability of $\hat \lambda$ may lead to excessive variability in $\hat p_n$. 
We therefore wish to find middle-grounds between the SB approximation and $\hat p_n$. 
This amounts to using the consistent estimates 
$\hat \lambda_j$ to calculate new weights that may reduce the sample variability of $p_n$, and 
at the same time reduce the effect of inconsistency in  SB. 
Consider for instance the following split-half approximation, where the lower half of the eigenvalues are replaced by their mean value, and likewise for the upper half of the eigenvalues.  
\begin{equation*} 
\hat p_{n, half} =P \left( \sum_{j=1}^d \tilde \lambda_j Z_j^2 > T_n \right),
\end{equation*}
where
\[
\tilde \lambda_1 = \cdots = \tilde \lambda_{\lceil d/2 \rceil} = \frac{1}{\lceil d/2 \rceil} \sum_{j=1}^{\lceil d/2 \rceil} \hat \lambda_j
\]
and 
\[
\tilde \lambda_{\lceil d/2 \rceil+1} = \cdots = \tilde \lambda_{d} = \frac{1}{d - \lceil d/2 \rceil} \sum_{j=\lceil d/2 \rceil+1}^{d} \hat \lambda_j.
\]
This procedure allows the p-value approximation an additional degree of freedom compared to the SB statistic, where all eigenvalues are estimated to be equal to each other.
A whole class of middle-grounds between the full $\hat p_n$ and $\hat p_{n, \text{SB}}$ can be defined as follows. Choose cut-off integers $1 < \tau_1 < \tau_2 < \cdots < \tau_k < d$ with $1 \leq k < d$. For $\tau_{l-1} \leq k < \tau_{l}$ let
\begin{equation} \label{equ::tildelambda}
\tilde \lambda_k = \frac{1}{\tau_{l}-\tau_{l-1}} \sum_{j=\tau_{l-1}}^{\tau_{l}-1} \hat \lambda_j
\end{equation}
where $\tau_0 = 1$ and $\tau_{k+1} = d$.
Let us denote this choice by $\tilde \lambda(\tau)= (\tilde \lambda_1(\tau), \ldots, \tilde \lambda_r(\tau))'$. The proposed $p$-value estimator is then
\[
\hat p_n(\tau) = P \left( \sum_{j=1}^d \tilde \lambda_j (\tau) Z_j^2 > T_n \right).
\]

An extension of the above framework is tests that assess nested hypotheses in SEM. Due to its great practical importance, we here include a short discussion on this special case. We again focus on the statistic $T_n$, since this statistic is typically asymptotically equivalent to other tests of interests, as described in \citet{satorra1989alternative}.

Following \citet{satorra1989alternative}, let $H : \sigma = \sigma(\theta), \theta \in \Theta$ and $H_0 : \sigma = \sigma(\theta), \theta \in \Theta_0$ where $\Theta_0 = \{ \theta \in \Theta : a(\theta) = 0 \}$ for some continuously differentiable function $a$. We assume that the matrix $\frac{\partial a (\theta)}{\partial \theta}$ has full row rank, say $m$. We let
\[
\hat \theta = \argmin_{\theta \in \Theta} F(s, \sigma(\theta)), \qquad \tilde \theta = \ \argmin_{\theta \in \Theta_0} F(s, \sigma(\theta))
\]
and $T_n = n F(s, \sigma(\hat \theta)  )$ and $\tilde T_n = n F(s, \sigma(\tilde \theta))$. 
Under $H_0$ and the conditions of Lemma 1 (iv) in \citet{satorra1989alternative} we have 
\begin{align*}
T_n &= \sqrt{n} (s - \sigma^\circ)' U \sqrt{n} (s - \sigma^\circ) + o_P(1) \\
\tilde T_n &= \sqrt{n} (s - \sigma^\circ)' \tilde{U} \sqrt{n} (s - \sigma^\circ) + o_P(1),
\end{align*}
for matrices $U$ and $\tilde{U}$ following the formula of eq.~\eqref{equ::Udef} under $H$ and $H_0$,  respectively.  Using the basic algebraic fact that $x' (A+B)x = x'Ax + x'Bx$ we conclude that the difference statistic is of the form
\[
\tilde T_n - T_n = \sqrt{n} (s - \sigma^\circ)' U_d \sqrt{n} (s - \sigma^\circ) + o_P(1), 
\]
where $U_d =   \tilde{U} - U$ has rank $m$. 

By the continuous mapping theorem, the convergence $\sqrt{n} (s - \sigma^\circ)  \ndconv N(0, \Gamma)$, and Theorem 1 in \cite{Box}, we therefore have that
\begin{equation} \label{equ::diffTest}
\tilde T_n - T_n \ndconv \sum_{j=1}^m \alpha_j Z_j^2, \qquad Z_1, \ldots, Z_m \sim N(0,1) \text{ IID},
\end{equation}
where $\alpha_1, \ldots, \alpha_m$ are the $m$ non-zero eigenvalues of $U_d \Gamma$.

Distribution-free consistent estimators $\hat U_d$ and $ \hat \Gamma$ for $U_d$ and $ \Gamma$ are found and discussed in \citet{satorra2001scaled}, and we do not review them here. Again the standard estimators can be found in software such as the \textsf{R} package lavaan. One then forms $\hat \alpha = (\hat \alpha_1, \ldots, \hat \alpha_m)'$ equal to the $m$ largest eigenvectors of $\hat U_d \hat \Gamma$ and calculates the full p-value approximation
\[
\hat p_n = P \left( \sum_{j=1}^m \hat \alpha_j Z_j^2 > \tilde T_n - T_n\right).
\]

We remark that a single equality constraint, say, $\beta_{i,j} = 0$, can be treated as special case of the above framework. In this case, the number of restrictions is $1$, and hence the limiting distribution in eq.~\eqref{equ::diffTest} is a scaled $\chi_1^2$. The SB and the proposed $p$-value approximations then coincide exactly. 
Note that Theorem \ref{theorem::consistency} implies that these procedures are consistent.

\section{A selection algorithm for p-value approximations} \label{section::selection}

The framework of the last section leads to several competing p-value approximations, 
and we next introduce a way of selecting among these. Our selector is inspired by \citet{beran1985bootstrap}, the Bollen-Stine bootstrap \citep{bollen1992bootstrapping}, and the non-parametric focused information criterion of \citet{jullum2016}. 

We wish to select the p-value approximation $\hat p_n$ whose distribution is closest to the uniform distribution under the null hypothesis. We formalize this by estimating the supremum distance between the cumulative distribution function of $\hat p_n$ under the null hypothesis and the uniform distribution, i.e. we approximate
\[
D_n = \sup_{0 \leq x \leq 1} |P_{H_0}(\hat p_n \leq x) - x|
\]
for each p-value approximation, and select the method with the least value of $D_n$. The probability $P_{H_0}$ is the probability measure induced by the data-generating distribution that is closest to fulfilling the null hypothesis compared to the true data-generating mechanism, which we let be the data generating distribution of $\Sigma(\theta^\circ)^{1/2} \Sigma^{-1/2} X_i$, where $\Sigma$ is the true covariance matrix. Under $P_{H_0}$, we know that p-values should be uniformly distributed. If we consider asymptotically consistent p-values, minimizing $D_n$ will mean that we choose the approximation whose convergence has been best achieved at our sample-size $n$.

The approximation to $D_n$ is done via the non-parametric bootstrap, based on the transformed sample
$\tilde X_i = \Sigma(\hat \theta)^{1/2} S_n^{-1/2} X_i$ for $i= 1, 2, \ldots, n$, as described in Algorithm \ref{alg::select}. The supremum in Algorithm \ref{alg::select} is the test statistic of the Kolmogorov-Smirnov test, which is implemented in most statistical software packages.
Formally, what we do is to use the empirical distribution function $\hat P_n$ of $(\tilde X_i)$ as an approximation to $P_{H_0}$, and approximate this probability through re-sampling. We then plug this approximation into $D_n$ to generate $\hat D_n$ for each p-value approximation. 

We note that we may use this selector among any p-value approximation for hypothesis testing in moment structures, and not just the suggestions in Section \ref{section::testprocedure}. 
Also,  $D_n$ is only one out of many possible success criteria. One could also investigate the mean square error of the approximation, or the distance from $P_{H_0}(\hat p_n \leq x) $ to $x$ at a particular point $x$. In our simulations, the performance of $D_n$ as a selection criterion was overall satisfactory.

\begin{algorithm*}[ht!]
	\caption{Selection algorithm}\label{alg::select}
	\begin{algorithmic}[1]
		\Procedure{Select}{sample, B}
		\State $\tilde X_i = \Sigma(\hat \theta)^{1/2} S_n^{-1/2} X_i$ for $i = 1, 2, \ldots, n$.
		\For{$ k \gets 1, \ldots, B$}
		\State boot.sample $\gets$ Draw with replacement from transformed sample $\tilde X_i$
		\For {$ l \in 1, \ldots, L$}
		\State $\hat p_{n,l} \gets$ based on boot.sample
		\EndFor
		\EndFor
		\For {$ l \in 1, \ldots, L$}
		\State $\hat D_{B,n,l} \gets \sup_{0 \leq x \leq 1} |B^{-1} \sum_{k=1}^{B} I \{ \hat p_{n,l} < x \} - x|$
		\EndFor
		\State \textbf{return } $\argmin_{1 \leq l \leq L} \hat D_{B,n,l}$
		\EndProcedure
	\end{algorithmic}
\end{algorithm*}

\section{Hypothesis tests for Satorra-Bentler consistency\\ and asymptotic robustness}\label{section::ar}

In this section we propose a bootstrap procedure for testing Satorra-Bentler consistency, that is, that all non-zero eigenvalues are equal. 
This also leads naturally to a test for asymptotic robustness (AR), that is, that all non-zero eigenvalues are equal to $1$. Such tests may help a practitioner to decide whether it is advisable to apply the NTML test, the SB test, or to instead use the Bollen-Stine bootstrap or the new procedures proposed in the present article. 

There is a substantial body of theoretical literature on AR \citep{Shapiro87, Browne88, Amemiya:1990wi, satorra1990model}, where
exact conditions are given that involve certain relationships  between $\Gamma$ and $\Delta$ that must hold for $\mlt$ to retain its asymptotic chi-square distribution under non-normality. However, these conditions 
are hard to check in practice, and currently no practical procedure exist for
verifying asymptotic robustness  in a real-world setting \cite[p. 118]{Yuan:2005p1088}.
Similarly, we unaware of the existence of tests for SB consistency. 
This lack of tests might be due to the fact  
 that testing statements concerning the eigenvalues of $U \Gamma$ involves testing statements about high moment properties of a distribution. Without detailed parametric assumptions on the data  it seems very difficult to construct tests that perform well in small-sample situations. It is therefore expected that our proposed procedures will require a large sample size to attain Type I error rates close to the nominal level. 
This is confirmed to be the case in the simulation experiment in Section \ref{section::bootstrapSim}.

The proposed bootstrap test is summarized in Algorithm \ref{alg::bootstrapSB} and is inspired by Section 4.2 in \citet{beran1985bootstrap}. A proof of its consistency, which we do not provide, seems to require a non-trivial extension of the theory contained in \citet{beran1984bootstrap, beran1985bootstrap}. 
We note that an important difference between our suggested test and the procedures in Section 4.2 in \citet{beran1985bootstrap}, who work with eigenvalues of empirical covariance (i.e., symmetric) matrices, is that $U \Gamma$ is typically not symmetrical.

Let $E$ be the matrix of normalised (complex) eigenvectors sorted by descending values of the eigenvalues $\lambda$  of $U \Gamma$. We have
$
U \Gamma = E \Delta E^\inv
$
\citep[p.514]{meyer2000matrix}
where
$
\Delta = \begin{pmatrix} \Delta_d & 0 \\ 0 & 0 \end{pmatrix},
$
and where $\Delta_d$ is the diagonal matrix with elements $\lambda_1, \ldots, \lambda_d$. Define
\begin{equation} \label{eq::A}
A = c^{1/2} \cdot E \begin{pmatrix} \Delta_d^{-1/2} & 0 \\ 0 & 0 \end{pmatrix} E^\inv,
\end{equation}
where $c$ denotes the mean value of the eigenvalues $\lambda_1, \ldots, \lambda_d$.
We propose the following bootstrap procedure. Let $\hat A$ 
be estimated from the original sample, by replacing $E, \Delta, c$ with $\hat E, \hat \Delta, \hat c$. 
For each bootstrap sample drawn from the original sample, we 
calculate
$\hat U_\text{boot} \hat \Gamma_\text{boot}$ and form the matrix
\[
W_n^* =  \hat A \hat U_\text{boot} \hat \Gamma_\text{boot} \hat A.
\]
The crucial observation is now that  $W_n^*$ converges to a matrix for which the null-hypothesis is true, that is, whose non-zero eigenvalues are all equal. To see this, note that
\begin{align*}
W_n^* & \npconv c E \begin{pmatrix} \Delta_d^{-1/2} & 0 \\ 0 & 0 \end{pmatrix} E^\inv \cdot U \Gamma \cdot E \begin{pmatrix} \Delta_d^{-1/2} & 0 \\ 0 & 0 \end{pmatrix} E^\inv  \\ &=  c E \begin{pmatrix}  \Delta_d^{-1/2} & 0 \\ 0 & 0 \end{pmatrix} E^\inv E \Delta E^\inv E \begin{pmatrix} \Delta_d^{-1/2}   & 0 \\ 0 & 0 \end{pmatrix} E^\inv \\
&= cE \begin{pmatrix}  \Delta_d^{-1/2} & 0 \\ 0 & 0 \end{pmatrix} \begin{pmatrix} \Delta_d & 0 \\ 0 & 0 \end{pmatrix} \begin{pmatrix} \Delta_d^{-1/2}   & 0 \\ 0 & 0 \end{pmatrix} E^\inv  = E \begin{pmatrix} c I_d & 0 \\ 0 & 0 \end{pmatrix} E^\inv,
\end{align*}
where the last matrix has $d$ non-zero eigenvalues equal to $d$. 
In the bootstrap sample, 
the $d$ largest eigenvalues of $W_n^*$ is then computed as $\hat \lambda_\text{boot}$. This process is repeated many times, and we get realizations $\hat \lambda_{k, \text{boot}}$, 
giving us information about the sampling variability of the estimated eigenvalues under the null hypothesis of identical eigenvalues. 

The above procedure may also be adapted to test for asymptotic robustness of the NTML statistic $\mlt$, that is, whether  $\lambda_j=1$ for all $j=1,\ldots, d$. 
 By setting $c=1$ in eq.~\eqref{eq::A}, Algorithm \ref{alg::bootstrapSB} then produces a p-value for the test of consistency of $\mlt$ based testing. Since this test does not need to estimate $c$, it should converge to the correct level I error rate slightly faster than the general case. 
The test statistic that is bootstrapped is then
\[
W_n^* =  \hat A_1 \hat U_\text{boot} \hat \Gamma_\text{boot} \hat A_1.
\]
where $\hat A_1$ is the estimator of $A_1 = E \begin{pmatrix} \Delta_d^{-1/2} & 0 \\ 0 & 0 \end{pmatrix} E^\inv$.

We suppose that an extension of Corollary 4 in \citet{beran1985bootstrap} holds also in our setting. That corollary requires the test statistic $h(\lambda)$ to be non-negative and zero under the null hypothesis, and that it has partial derivatives that are zero under the null hypothesis and that its double derivative matrix is positive definite under the null hypothesis. The additional restriction that also the partial derivatives vanish under the null-hypothesis means we must consider two different test statistics, adapted from the examples in Section 4.3 in \citet{beran1985bootstrap}.
For asymptotic robustness, this holds for $h_{AR}(\lambda) = d \log [d^\inv \sum_{j=1}^d \lambda_j] - \log [\prod_{j=1}^d \lambda_j]$.
For Satorra-Bentler consistency, this holds for $h_{SB}(\lambda) = \log [(\lambda_1 + \lambda_d)^2] - \log [4 \lambda_1 \lambda_d].$
Algorithm \ref{alg::bootstrapSB} summarizes this discussion.

\begin{algorithm*}
	\caption{Bootstrap testing for Satorra-Bentler consistency and Asymptotic Robustness} \label{alg::bootstrapSB}
	\begin{algorithmic}[1]
		\Procedure{Bootstrap}{sample, B}
		\State Calculate $\hat U,  \hat \Gamma, \hat A, \hat A_1$	from sample
		\State $\hat \lambda \gets$ The $d$ largest eigenvalues of $\hat U \hat \Gamma$	
		\State $T_{n, \text{SB}} = h_{SB}(\hat \lambda)$
		\State $T_{n, \text{AR}} = h_{AR}(\hat \lambda)$
		\For{$ k \gets 1, \ldots, B$}
		\State boot.sample $\gets$ Draw with replacement from sample
		\State $\hat U_\text{boot} \hat \Gamma_\text{boot} \gets$ Based on boot.sample
		\State $W_{n, \text{SB}}^*\gets \hat A \hat U_\text{boot} \hat \Gamma_\text{boot} \hat A$
		\State $\hat \lambda_{k,\text{boot}} = (\hat \lambda_{k,1,\text{boot}}, \ldots, \hat \lambda_{k,d,\text{boot}})' \gets$ the $d$ largest eigenvalues of $W_{n, \text{SB}}^*$  
		\State $T_{n, k, \text{SB}} \gets h_{SB}(\hat \lambda_{k,\text{boot}})$
		\State $W_{n, \text{AR}}^*\gets \hat A_1 \hat U_\text{boot} \hat \Gamma_\text{boot} \hat A_1$
		\State $\hat \lambda_{k,\text{boot}} = (\hat \lambda_{k,1,\text{boot}}, \ldots, \hat \lambda_{k,d,\text{boot}})' \gets$ the $d$ largest eigenvalues of $W_{n, \text{AR}}^*$  				
		\State $T_{n, k, \text{AR}} \gets h_{AR}(\hat \lambda_{k,\text{boot}})$
		\EndFor
		\State \textbf{return } $B^\inv \sum_{k=1}^B I \{ T_{n, k, \text{SB}} > T_{n, \text{SB}} \}$ and $B^\inv \sum_{k=1}^B I \{ T_{n, k, \text{AR}} > T_{n, \text{AR}} \}$
		\EndProcedure
	\end{algorithmic}
\end{algorithm*}

\section{Monte Carlo evaluations} \label{section::montecarlo}

In this section we evaluate the proposed procedures by Monte Carlo methods. We first evaluate our new class of p-value approximations in the setting of goodness-of-fit testing for a single model. Specifically, two members of this class are evaluated, $\hat p_n$ and $\hat p_{n, \text{half}}$, referred to as the full and half eigenvalue approximations, respectively. 
We then consider chi-square difference testing for two nested models. In both cases we  evaluate the selection procedure in Algorithm 
\ref{alg::select}
where the candidates for selection are SB and the full and half eigenvalue approximations. 
Finally, we evaluate the empirical performance of the proposed bootstrap test for SB consistency and for AR. 
We remark that  we here limit ourselves  to study the empirical performance of the procedures when it comes to controlling type I error rates, leaving  the topic of power for future studies.

Our model is the  political democracy model 
discussed by Bollen in his textbook  \citep{bollen},  see Figure \ref{fig::bollen},  where the residual errors are not depicted for ease of presentation. There are four measures of political democracy measured twice (in 1960 and 1965),  and three measures of industrialization measured once (in 1960). The unconstrained model $\mathcal{M}_1$ has $d=35$ degrees of freedom. For nested model testing, we also consider a constrained model $\mathcal{M}_0$, nested within $\mathcal{M}_1$, 
with $d=46$ degrees of freedom,  which impose ten equalities among unique variances
and residual covariances, and one equality between two factor loadings.

Model estimation and eigenvalues were computed using the \textsf{R} package lavaan \citep{rosseel2012lavaan}, while
the p-values of type $\hat p_n$ were calculated with the imhof procedure in the package CompQuadForm \citep{Duchesne:2010fk}. 
In each simulation cell we 
generated $2000$ samples. 
For each sample, $1000$ bootstrap samples were drawn.

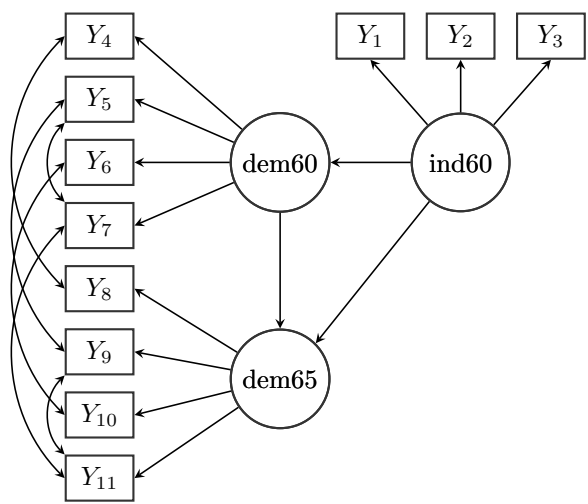
\begin{figure}[ht]
	\caption{Bollen's political democracy model. dem60: Democracy in 1960. dem65: Democracy in 1965. ind60: Industrialisation in 1960.}
	\centering
	
	\begin{tikzpicture}[>=stealth,semithick,scale=0.6]
	\node[ov] (y1) at (0,0)     { $Y_4$};
	\node[ov] (y2) at (0,-1.4)  { $Y_5$};
	\node[ov] (y3) at (0,-2.8)  {$Y_6$};
	\node[ov] (y4) at (0,-4.2)  { $Y_7$};
	\node[ov] (y5) at (0,-5.6)  { $Y_8$};
	\node[ov] (y6) at (0,-7)    { $Y_9$};
	\node[ov] (y7) at (0,-8.4)  { $Y_{10}$};
	\node[ov] (y8) at (0,-9.8)  { $Y_{11}$};
	\node[ov] (x1) at ( 6,0)     { $Y_1$};
	\node[ov] (x2) at ( 8,0)     { $Y_2$};
	\node[ov] (x3) at (10,0)     { $Y_3$};
	\node[lv] (f1) at (4, -2.8) { dem60};
	\node[lv] (f2) at (4, -7.6) { dem65};
	\node[lv] (f3) at (8, -2.8) { ind60};
	\node[lv] (f1) at (4, -2.8) {dem60};
	\node[lv] (f2) at (4, -7.6) {dem65};
	\node[lv] (f3) at (8, -2.8) {ind60};
	\path[->] (f1) edge node[above=0.08cm,scale=0.8,pos=0.7] {} (y1.east);
	\path[->] (f1) edge node[above=0.08cm,scale=0.8,pos=0.7] {} (y2.east);
	\path[->] (f1) edge node[above=0.08cm,scale=0.8,pos=0.7] {} (y3.east);
	\path[->] (f1) edge node[above=0.08cm,scale=0.8,pos=0.7] {} (y4.east);
	\path[->] (f2) edge node[above=0.08cm,scale=0.8,pos=0.7] {} (y5.east);
	\path[->] (f2) edge node[above=0.08cm,scale=0.8,pos=0.7] {} (y6.east);
	\path[->] (f2) edge node[above=0.08cm,scale=0.8,pos=0.7] {} (y7.east);
	\path[->] (f2) edge node[above=0.08cm,scale=0.8,pos=0.7] {} (y8.east);
	\path[->] (f3) edge node[above=0.08cm,scale=0.8,pos=0.7] {} (x1.south);
	\path[->] (f3) edge node[above=0.08cm,scale=0.8,pos=0.7] {} (x2.south);
	\path[->] (f3) edge node[above=0.08cm,scale=0.8,pos=0.7] {} (x3.south);
	\path[->] (f1) edge node[above=0.08cm,scale=0.8,pos=0.7] {} (f2);
	\path[->] (f3) edge node[above=0.08cm,scale=0.8,pos=0.7] {} (f1.east);
	\path[->] (f3) edge node[above=0.08cm,scale=0.8,pos=0.7] {} (f2.north east);
	\path[<->] (y1.west) edge [bend right=45] node[left,scale=0.8] {} (y5.west);
	\path[<->] (y2.west) edge [bend right=45] node[left,scale=0.8] {} (y6.west);
	\path[<->] (y3.west) edge [bend right=45] node[left,scale=0.8] {} (y7.west);
	\path[<->] (y4.west) edge [bend right=45] node[left,scale=0.8] {} (y8.west);
	\path[<->] (y2.south west) edge [bend right=45] node[left,scale=0.8] {} (y4.north west);
	\path[<->] (y6.south west) edge [bend right=45] node[left,scale=0.8] {} (y8.north west);
	\end{tikzpicture}
	
	\label{fig::bollen}
\end{figure}

\subsection{Goodness-of-fit testing for $\mathcal{M}_1$ }
\begin{table}[ht]
	\centering
	\begin{tabular}{lrrrrrrr}
		\multirow{5}{*}{Distribution 2}	&	1.87 & 1.59 & 1.49 & 1.44 & 1.43 & 1.42 & 1.38 \\ 
		&1.36 & 1.35 & 1.34 & 1.31 & 1.29 & 1.26 & 1.13 \\ 
		&1.12 & 1.11 & 1.11 & 1.10 & 1.10 & 1.09 & 1.09 \\ 
		&1.08 & 1.08 & 1.07 & 1.07 & 1.07 & 1.06 & 1.05 \\ 
		&1.04 & 1.03 & 1.03 & 1.03 & 1.02 & 1.02 & 1.01 \\ \hline
		\multirow{5}{*}{Distribution 3}	&4.16 & 3.24 & 2.88 & 2.82 & 2.70 & 2.67 & 2.51 \\ 
		&	2.41 & 2.35 & 2.31 & 2.16 & 2.12 & 2.03 & 1.52 \\ 
		&1.50 & 1.47 & 1.43 & 1.40 & 1.38 & 1.36 & 1.35 \\ 
		&	1.33 & 1.32 & 1.29 & 1.27 & 1.25 & 1.21 & 1.20 \\ 
		&	1.13 & 1.13 & 1.11 & 1.09 & 1.08 & 1.08 & 1.06 
	\end{tabular}
	\caption{Eigenvalues $\lambda_i$, for $i=1, \ldots, 35$,  for Bollen's political democracy model, assuming   correct model specification. Distribution 2 and 3 have univariate skewness and kurtosis $s=1, k=7$ and $s=2, k=21$, respectively.}
	\label{tab::eigen}
\end{table}

\begin{table}[ht]
	\centering
	\begin{tabular}{llrrrrrrrr}
		Distribution & n & NTML & SB & SS & BOST & EFULL & EHALF & SEL & ORAC \\ \hline
		\multirow{3}{*}{Normal} & 100 & 0.077 & 0.086 &   0.050 & 0.023 & 0.036 & 0.050 & 0.051 & 0.077 \\ 
		& 300 & 0.055 & 0.053 &  0.052  &0.037 & 0.037 & 0.043 & 0.045 & 0.055 \\ 
		& 900 & 0.068 & 0.067 & 0.050  &   0.059 & 0.063 & 0.064 & 0.065 & 0.068 \\ \hline
		\multirow{3}{*}{Distribution 2} & 100 & 0.215 & 0.108 & 0.019  & 0.035 & 0.021 & 0.048 & 0.042 & 0.057 \\ 
		& 300 & 0.197 & 0.070 &  0.018  & 0.053  &0.024 & 0.045 & 0.045 & 0.057 \\ 
		& 900 & 0.219 & 0.063 &    0.033& 0.054 & 0.037 & 0.051 & 0.051 & 0.059 \\ \hline
		\multirow{3}{*}{Distribution 3} & 100 & 0.488 & 0.164 & 0.017   & 0.038 & 0.009 & 0.072 & 0.031 & 0.024 \\ 
		& 300 & 0.591 & 0.094 &   0.013& 0.068 & 0.013 & 0.050 & 0.038 & 0.045 \\ 
		& 900 & 0.685 & 0.076 &    0.017 & 0.059 & 0.015 & 0.042 & 0.038 & 0.046 \\ \hline
	\end{tabular}
	\caption{Type I error rates for testing model $\mathcal{M}_1$.
		Normal: multivariate normal distribution, Distribution 2: skewness $1$ and kurtosis $7$. Distribution 3: skewness $2$ and kurtosis $7$. NTML=normal-theory likelihood ratio test. SB=Satorra-Bentler. SS=Scaled and shifted.  BOST=Bollen-Stine bootstrap. EFULL= Full eigenvalue approximation, $\hat p_n$. EHALF= half eigenvalue approximation, $\hat p_{n,\text{half}}$. 
		SEL = p-value obtained from selection algorithm. ORAC= oracle p-value  $p_n$. }
	\label{tab::nonnestedNEW}
\end{table}

\begin{table}[ht]
	\centering
	\begin{tabular}{llrrr}
		Distribution & $n$ & SB & EHALF & EFULL \\ 
		\hline
		\multirow{3}{*}{Normal} & 100 & 0.054 & 0.931 & 0.015 \\ 
		& 300 & 0.448 & 0.516 & 0.036 \\ 
		& 900 & 0.507 & 0.263 & 0.231 \\  \hline
		\multirow{3}{*}{Distribution 3} & 100 & 0.001 & 0.865 & 0.135 \\ 
		& 300 & 0.050 & 0.894 & 0.055 \\ 
		& 900 & 0.153 & 0.783 & 0.063 \\ \hline
		\multirow{3}{*}{Distribution 3} & 100 & 0.000 & 0.449 & 0.551 \\ 
		& 300 & 0.001 & 0.733 & 0.267 \\ 
		& 900 & 0.004 & 0.846 & 0.150 \\ 
		\hline
	\end{tabular}
	\caption{Choice proportions for selection algorithm,  testing model $\mathcal{M}_1$. SB=Satorra-Bentler. EHALF=half eigenvalue approximation, $\hat p_{n,\text{half}}$. EFULL= Full eigenvalue approximation, $\hat p_n$.  }
	\label{tab::nonnestedcounts}
\end{table}

Three population distributions were considered.
Distribution 1 was a multivariate normal distribution. The non-normal distributions were generated using the transform of \cite{vale1983simulating}, with Distribution 2 having univariate skewness $1$ and kurtosis $7$, and Distribution 3 having skewness $2$ and kurtosis $21$.
These distributional characteristics are the same as those used in the influential study by \cite{Curran}, and  replicated in the bootstrap study by \cite{Nevitt:2001gg}.
 The "oracle" eigenvalues associated with Distribution 2 and 3 are given in Table \ref{tab::eigen},
numerically calculated from  very large samples, where we clearly see that the values 
are quite spread out, and that the spread increases when we  
move from Distribution 2 to Distribution 3. Note that under the Distribution 1, we have $\lambda = (1, 1, \ldots, 1)'$.

Three sample sizes $n$ were used: $100, 300$ and $900$. Hence the resulting full factorial design has nine conditions. In each sample we calculated p-values associated with the established test statistics associated with normal-theory maximum likelihood ratio (NTML), Satorra-Bentler scaling (SB), the scale-and-shifted statistic (SS) and the Bollen-Stine (BOST) test. 
Also, we calculated in each sample the full eigenvalue approximation $\hat p_n$ (EFULL) and the  split-half eigenvalue estimation $\hat p_{n, \text{half}}$ (EHALF).
The selection algorithm (SEL) p-value was calculated using a candidate set   with members  SB, EHALF and EFULL,  and using  $\hat D_n$ as criterion function. Finally, the oracle (ORAC) p-value $p_n$ was calculated, using the values in Table \ref{tab::eigen}.  
This allows us to 
evaluate  how well the asymptotic result in eq.\eqref{mixture} applies in  finite-sample conditions.

In Table \ref{tab::nonnestedNEW} we present Type I error rates at the the $5 \%$ significance level. 
As expected, NTML becomes inflated when data is non-normal. The mean-scaling of SB reduces the inflation, but
with  non-normal data and small sample sizes,  Type I error rates are still higher than $10\%$.
The scaled-and-shifted statistic on the other hand, leads to rejection rates much lower than the nominal  $5 \%$.
These findings are in accord with \cite{foldnes2015correcting}. The Bollen-Stine bootstrap test performs better than
SB and SS, coming close to the nominal level even for highly non-normal data and medium sample size. 
Among the new p-value approximations, it is the middle-ground approximation EHALF that performs the best.
While EFULL yields far too low rejection rates with non-normal data. EHALF as well as BOST with non-normal data. 
The selection algorithm SEL also performs generally well, on par with EHALF and BOST. 
It is notable that for normal data, SEL outperforms NTML.
Table \ref{tab::nonnestedcounts} presents the selection proportions for SEL in each of the nine conditions. 
It is seen that the selection algorithm wisely chooses EHALF in the majority of conditions. It is however unexpected that
SEL chooses EFULL in $55\%$ of the samples under Distribution 3 and $n=100$, given the poor performance of EFULL
in that condition, with a $1\%$ rejection rate. 
The final column in Table \ref{tab::nonnestedNEW} gives the oracle solution, and demonstrates that the asymptotic result
in \eqref{mixture} is far from realized at $n=100$  under Distribution 3.  

\subsection{Testing nested models.}

The chi-square difference test has $11$ degrees of freedom, and the corresponding $11$  
oracle eigenvalues for Distribution 2 and 3 are given in Table \ref{tab::eigendiff}.
The spread in eigenvalues is substantial, especially for Distribution 3. 
\begin{table}[ht]
	\centering
	\begin{tabular}{rrrrrrrrrrrr}
		Distribution 2 & 	3.92 & 3.49 & 3.19 & 2.99 & 2.94 & 2.78 & 2.72 & 1.85 & 1.56 & 1.54 & 1.30 \\ 
		Distribution 3 & 	10.64 & 8.79 & 8.06 & 7.58 & 7.37 & 6.94 & 6.76 & 4.09 & 3.16 & 3.10 & 2.04 \\   
	\end{tabular}
	\caption{Eigenvalues of $U_d \Gamma$ for nested model testing.   Distribution 2 has skewness $1$ and kurtosis $7$;  Distribution 3 has skewness $2$ and kurtosis $21$. Rounded to two decimal places.}
	\label{tab::eigendiff}
\end{table}

\begin{table}[ht]
	\centering
	\begin{tabular}{llrrrrrrr}
		Distribution & $n$ & ML & SB & BOST & EFULL & EHALF & SEL & ORAC \\ 
		\hline
		\multirow{3}{*}{Normal} & 100 & 0.068 & 0.080 & 0.037 & 0.062 & 0.069 & 0.075 & 0.068 \\ 
		& 300 & 0.054 & 0.059 & 0.046 & 0.053 & 0.055 & 0.058 & 0.054 \\ 
		& 900 & 0.051 & 0.053 & 0.051 & 0.051 & 0.052 & 0.053 & 0.051 \\ \hline
		\multirow{3}{*}{Distribution 2} & 100 & 0.582 & 0.137 & 0.096 & 0.076 & 0.099 & 0.096 & 0.028 \\ 
		& 300 & 0.659 & 0.088 & 0.081 & 0.052 & 0.066 & 0.062 & 0.035 \\ 
		& 900 & 0.702 & 0.059 & 0.053 & 0.035 & 0.043 & 0.045 & 0.046 \\ \hline
		\multirow{3}{*}{Distribution 3} & 100 & 0.911 & 0.221 & 0.129 & 0.115 & 0.159 & 0.135 & 0.005 \\ 
		& 300 & 0.961 & 0.126 & 0.118 & 0.062 & 0.089 & 0.082 & 0.018 \\ 
		& 900 & 0.976 & 0.087 & 0.089 & 0.044 & 0.064 & 0.061 & 0.043 \\ 
		\hline
	\end{tabular}
	\caption{Type I error rates for nested model testing. 
		Normal: multivariate normal distribution, Distribution 2: skewness $1$ and kurtosis $7$. Distribution 3: skewness $2$ and kurtosis $7$. NTML=normal-theory likelihood ratio test. SB=Satorra-Bentler.   BOST=Bollen-Stine bootstrap. EFULL= Full eigenvalue approximation, $\hat p_n$. EHALF= half eigenvalue approximation, $\hat p_{n,\text{half}}$. 
		SEL = p-value obtained from selection algorithm. ORAC= oracle p-value  $p_n$. }
	\label{tab::nestednew}
\end{table}

Rejection rates observed at the nominal $5\%$ level of signficance are reported in Table \ref{tab::nestednew}.
Again, the NTML statistic is inflated by non-normality in the data, a tendency only partially corrected for by SB.
For instance, under the most harsh condition, with Distribution 3 and $n=100$, SB rejection rates are $22\%$, far better
than the $91\%$ obtained with NTML. But in this condition, as in all conditions, BOST  performs better, with a rejection rate of $13\%$. However, the new procedure EFULL performs still better in this  condition, while the selection algorithm is only slightly worse than BOST. Overall EFULL outperforms the other test statistics, including SB and BOST. EHALF, which was found to have best performance in the non-nested case, does not
perform as well as EFULL in the nested case. The selection algorithm SEL also performs well, 
with better performance than SB and BOST in most conditions, 
and only sligthly worse then EFULL. The selection proportions are given in Table \ref{tab::nestedcounts}, 
where EHALF is unexpectedly found to be the most chosen procedure, despite the slightly better performance of EFULL in most conditions. 

\begin{table}[ht]
	\centering
	\begin{tabular}{llrrr}
		Distribution & $n$ & SB & EHALF & EFULL \\ 
		\hline
		\multirow{3}{*}{Normal} & 100 & 0.601 & 0.357 & 0.042 \\ 
		& 300 & 0.672 & 0.205 & 0.122 \\ 
		& 900 & 0.593 & 0.091 & 0.316 \\ \hline
		\multirow{3}{*}{Distribution 3} & 100 & 0.116 & 0.714 & 0.170 \\ 
		& 300 & 0.209 & 0.662 & 0.128 \\ 
		& 900 & 0.263 & 0.595 & 0.142 \\ \hline
		\multirow{3}{*}{Distribution 3} & 100 & 0.012 & 0.663 & 0.325 \\ 
		& 300 & 0.059 & 0.725 & 0.215 \\ 
		& 900 & 0.104 & 0.714 & 0.182 \\ 
		\hline
	\end{tabular}
	\caption{Choice proportions for selection algorithm,  nested models. SB=Satorra-Bentler. EHALF=half eigenvalue approximation, $\hat p_{n,\text{half}}$. EFULL= Full eigenvalue approximation, $\hat p_n$.  }
	\label{tab::nestedcounts}
\end{table}

\subsection{Tests for AR and for SB consistency} \label{section::bootstrapSim}

To evaluate Type I error rates of the SB consistency and AR tests proposed in Algorithm \ref{alg::bootstrapSB}, 
we simulated multivariate normal data for the Bollen model. Under normal data both AR and SB consistency holds. 
We simulated 2000 samples for sample sizes $n=200,400, 800$ and $2000$. For each sample $1000$ bootstrap samples were drawn. 
The rejection rates are given in Table \ref{tab::arsb}, and clearly demonstrates that these procedures
need large sample sizes in order to reach acceptable Type I error rates. 

\begin{table}[ht]
	\centering
	\begin{tabular}{lrrrr}
		Test & $n=200$ & $n=400$ & $n=800$ & $n=2000$ \\ 
		\hline
		AR & 0.354 & 0.203 & 0.081 & 0.035 \\ 
		SB & 0.369 & 0.195 & 0.070 & 0.033 \\ 
		\hline
	\end{tabular}
	\caption{Type I error rates for tests of asymptotic robustness (AR) and Satorra-Bentler (SB) consistency. }
	\label{tab::arsb}
\end{table}

\section{Discussion} 

This paper deals with the fundamental problem of hypothesis testing in moment structure models. We present new insight and practically applicable statistical methodology for SEM and related models.

Some of our conclusions may seem surprising, as they go against what is often taught in standard courses on SEM. For example, the simulation summarized in Table \ref{tab::nonnestedNEW} shows that our selector can have better finite sample performance than the NTML test also when data are exactly normal. Since this paper have focused exclusively on Type I error, ``better'' here means having a rejection rate closer to the nominal one.

Since this conclusion may seem counterintuitive, it is worth pausing and considering what the NTML does. Firstly, we must keep in mind that the NTML is a test based on asymptotic theory, also when data are exactly normally distributed. That is, the Type I error rate of an NTML test at level $\alpha$ converges to $\alpha$ under normality, and for the model considered in Table \ref{tab::nonnestedNEW}, convergence is still not quite achieved for $n = 100$. Secondly, we note that under normality,  NTML calculates the oracle p-value exactly. That is, it is the ultimate approximation to the oracle test, which has a rejection rate of 7.7 \%. Hence, the NTML has only one source of approximation error: the validity of the fundamental convergence of the oracle.

All methods considered in this paper -- with the important exception of the Bollen-Stine bootstrap and the selector -- tries to approximate the oracle, and thereby introducing another source of approximation error. Let us say they are oracle-based. Except the NTML, which calculates the oracle perfectly -- but only under exact normality, the oracle-based methods have varying degrees of success in their approximation. Strictly speaking, oracle-based methods should be judged on whether they manage to achieve what they set out to do: approximate the oracle. But that is not the success criteria of interests to the user: When a level $\alpha$ test is employed, the Type I error rate should be very close to $\alpha$. As is clear from our simulation studies, this may not be the case even when using the actual oracle. When oracle-based tests have Type I error rate considerably closer to $\alpha$ than the oracle, it is tempting to say that they are performing well. This temptation should be avoided, as the deviation in Type I error compared to the oracle is then solely due to chance variations caused by the estimation of $\lambda$.

The selector overcomes this hurdle by being anchored not in the fundamental convergence of the oracle, but by transforming the data to a setting where the null hypothesis holds. It is then known that a correct p-value is uniformly distributed, i.e., the Type I error rate of a test with level $\alpha$ is to be exactly $\alpha$. It is this anchoring that allows us to search for the procedure which best achieves this goal, without having to compare our methods to the finite sample performance of the oracle. And so when the selector has a Type I error rate close to the nominal, it is by design, and not solely due to chance variations. This is a property shared with the Bollen-Stine bootstrap procedure, but the Bollen-Stine procedure rests on the quality of the approximation of the empirical distribution function compared to the data's actual distribution function. So do we, since we use the non-parametric bootstrap in our selector, but we are able to combine the fundamental convergence of the oracle with the non-parametric bootstrap. We have seen that this allows us to combine the strengths of both methods.

Let us return to the NTML, and look at the proposed methods from a slightly different perspective that elaborates on the above. It is well-known that the NTML usually has a much too high Type I error rate under non-normality. The major source of the mismatch between nominal and actual Type I error rate is that the NTML need not be a consistent approximation to the oracle. The NTML can be seen as estimating $\lambda$ always by the constant $(1,\ldots, 1)'$. When $\lambda$ is far away from $(1,\ldots, 1)'$, NTML performs poorly. And for a user, it typically performs poorly in a particularly bad way: even when a hypothesized theory holds, the NTML will most likely reject it.

The Satorra-Bentler test has previously been reported to have inflated Type I error rates under non-normality, 
and this behaviour is also observed in simulations in the present paper. 
This is mainly due to two reasons: firstly, it may be that the Satorra-Bentler procedure is inconsistent, i.e. $\lambda$ has variation among its elements. While inconsistency is an asymptotic property, which may seem irrelevant in small samples, it does mean that the procedure does not aim to estimate what the user wants, and may therefore be reflected also in small-sample situations when the procedure is used uncritically. Secondly, the Satorra-Bentler procedure estimates $\lambda$, and the variability of the resulting p-value approximation may give inflated Type I errors even when the procedure is consistent.

These two problems, consistency and finite sample approximation error, are shared also by our suggested p-value approximations. However, the contextual framework presented in the present paper allows us to argue about balancing these issues, and selecting among competing approximations. This perspective may lead to further insight in future research, and has already led to our proposed selector.

We note that while $\hat p_n$ and $\hat p_{n, \text{half}}$ can be computed just as fast as the Satorra-Bentler test statistics, both the selector and the Bollen-Stine bootstrap procedure takes considerable more computation time.
Our simulation experiments indicate that the selector and the Bollen-Stine bootstrap are comparable in performance, but that the selector works slightly better, especially in small sample situations. 
Our recommendations to practitioners are therefore clear: use the selector in small sample situations, and use the selector or the Bollen-Stine bootstrap in medium sample situations. In large sample situations, consistent p-value approximation gives similar answers. Since the assumptions underlying asymptotic robustness and Satorra-Bentler consistency rests on delicate properties of high order moments that can only be properly tested in large sample situations, we do not recommend using the NTML nor the Satorra-Bentler statistic without assessing its performance with the selector. In many cases, the Satorra-Bentler statistic will be the superior test, but it is difficult to know this without using techniques such as the re-sampling methods underlying the selector.

With current and future computers containing multiple units that can perform computƒation simultaneously (multi-core central processing units and multi-core graphical processing units supporting general purpose calculations), using the selector does not take much time to run. In our prototype implementation in the scripting language R (which means our code is not compiled, and therefore slow), it takes a few additional minutes compared to standard p-value approximations that we have seen often performs considerably worse. Considering the enormous amount of time and effort many researchers use in gathering and analyzing data, the extra time spent on using the selector is vanishing in comparison. 

Applied researchers are often personally interested in controlling Type I error as well as possible, as their research hypothesis is often the null hypothesis. If they use testing procedures, such as the NTML with non-normal data, where the Type I error is seriously inflated, this is to their disadvantage.
This point is also connected to the  use of pragmatic fit indices available in the literature. The p-value approximations discussed in this paper are all based on solid statistical theory. The ad-hoc nature of some of these fit indices, with somewhat arbitrary cut-off points being interpreted in various ways, are not based on statistical theory. 

Finally, we mention that the ideas contained in this paper can be generalized in several directions, including SEM with ordinal variables and in multi-group settings.  Also, additional simulation experiments should be performed on the proposed methods, such as power studies, allowing the selector more options, and experimenting with different selection criterias.

 \newcommand{\noop}[1]{}

\begin{appendix}

	\section{Proof of Theorem \ref{theorem::consistency}} \label{section::proof}
	
	\begin{proof}[Proof of Theorem \ref{theorem::consistency}]
		By the mean value theorem, there is a sequence of random variables $0 \leq r_n \leq 1$ so that
$
			\hat p_{n} = 1 - H(T_n; \hat \lambda) = 1 - H(T_n; \lambda) + R_n = p_n + R_n
$
		where
		$
		R_n = \ \sum_{j=1}^d (\hat \lambda_j - \lambda_j) H_j(T_n; \lambda + r_n (\hat \lambda - \lambda))$ and $H_j(q, (l_1, \ldots, l_d)') = \frac{\partial H(q, (l_1, \ldots, l_d)'}{\partial l_j}.
		$
		The statement of the theorem therefore holds if we show that $H_j(T_n; \lambda + r_n (\hat \lambda - \lambda)) = O_P(1)$. 
		To show this, we calculate $H_j$. The cumulative distribution function of $S = \ \sum_{j=1}^{d} \lambda_j Z_j^2$ is $H_S(q) = \int_0^q h_S(s) \, \d s$ where $h_S$ is the density of $S$. 
		Denote the density of $\lambda_j Z_j^2$ by $h_{j}(z)$.
		Since $(Z_j)_{j=1}^d$ contains independent variables, so does  $(\lambda_j Z_j^2)_{j=1}^d$. Hence $h_S$ is given by $d$-times convolution, i.e.~apply the well-known convolution formula iteratively, and see that 
		\begin{equation} \label{equ::fSconvolution}
		g_S(s) = \int_{\mathbb{R}} \cdots \int_{\mathbb{R}} \left[ \prod_{j=2}^{d} h_j(x_{j-1}) \right] h_1 \left( s - \sum_{j=1}^{d-1} x_j \right) \, \d x_1 \cdots \, \d x_{r-1},
		\end{equation}
		see also \citet{laury1976n} for some basic properties of $d$-times convolution.
		We wish to calculate
		\begin{equation} \label{equ::FderivlambdaDef}
		\frac{\partial}{\partial \lambda_j} H_S(q) = \int_0^q \frac{\partial}{\partial \lambda_j} h_S(s) \, \d s.
		\end{equation}
		It turns out that $\frac{\partial}{\partial \lambda_j} h_S $ is a weighted sum of densities, which implies that $\frac{\partial}{\partial \lambda_j} H_S$ is a weighted sum of cumulative distribution functions that is easy to bound uniformly. We now show this by calculating $\frac{\partial}{\partial \lambda_j} h_S$.
		Since summation is commutative, the distribution of $\sum_{j=1}^r \lambda_{\pi(j)} Z_{\pi(j)}^2$ is the same for any permutation $\pi(1), \ldots, \pi(r)$ of $\{1, \ldots, r\}$. 
		We may therefore, without loss of generality, assume that $j = d$. Using eq.~\eqref{equ::fSconvolution}, we have 
		\begin{multline} \label{equ::convolutionDeriv}
		\frac{\partial}{\partial \lambda_d} h_S(s) \\ 
		= \int_{\mathbb{R}} \cdots \int_{\mathbb{R}} \left\{ \frac{\partial}{\partial \lambda_d} h_d(x_{d-1}) \right\} [\prod_{j=2}^{d-1} h_j(x_{j-1})] h_1(s - \sum_{j=1}^{d-1} x_j) \, \d x_1 \cdots \, \d x_{d-1}
		\end{multline}
		We hence need to calculate $\frac{\partial}{\partial \lambda_d} h_d(x_{d-1})$.
		Since $\lambda_j Z_j^2$ is a linear transformation of $Z_j^2 \sim \chi_1^2$, we have $h_{j}(z) = h_{\chi^2}(z/\lambda_j)/\lambda_j$ where
		$
		h_{\chi^2}(z) = \frac{z^{1/2}}{\sqrt{2 \pi}} e^{-z/2} I \{ z \geq 0 \}
		$
		is the density of $Z_j^2 \sim \chi^2$.
		
		We have
$
			\frac{\partial}{\partial \lambda_d} h_d(z)  = \frac{\partial}{\partial \lambda_d} \lambda_d^\inv h_{\chi^2}(z/\lambda_d) = - \lambda_d^{-2} h_{\chi^2}(z/\lambda_d) + \lambda_d^\inv h_{\chi^2}'(z/\lambda_d) [ \frac{\partial}{\partial \lambda_d} z \lambda_d^\inv ] 
			 = - \lambda_d^{-2} h_{\chi^2}(z/\lambda_d) - \lambda_d^{-3} z h_{\chi^2}'(z/\lambda_d).
$
		For $z < 0$ then $h_d(z) = 0$ and so $\frac{\partial}{\partial \lambda_d} h_d(z) = 0$. The event $z = 0$ has probability zero and can be ignored. For $z > 0$ we have
$
			\sqrt{2 \pi} h_{\chi^2}'(z)  =\sqrt{2 \pi} \frac{\d}{\d z} \frac{z^{1/2}}{\sqrt{2 \pi}} e^{-z/2} = \frac{\d}{\d z} z^{1/2} e^{-z/2}
			 = \half z^{-1/2} e^{-z/2} + (-\half) z^{1/2} e^{-z/2} 
$
		so that
		$
		h_{\chi^2}'(z/\lambda_d) = \half \lambda_d^{1/2} z^{-1/2} e^{-z/(2 \lambda_d)} -\half \lambda_d^{-1/2} z^{1/2} e^{-z/(2 \lambda_d)} 
		$
		Inserting this into the expression obtained for $\frac{\partial}{\partial \lambda_d} h_r(z)$ gives
$
			\frac{\partial}{\partial \lambda_d} h_d(z) = - \lambda_d^{-2} f_{\chi^2}(z/\lambda_d) - \lambda_d^{-3} z [ \half \lambda_d^{1/2} z^{-1/2} e^{-z/(2 \lambda_d)} -\half \lambda_d^{-1/2} z^{1/2} e^{-z/(2 \lambda_d)}  ] 
			 = - \lambda_d^{-1} \lambda_d^{-1} h_{\chi^2}(z/\lambda_d) - \half \lambda_d^{-5/2} z^{1/2} e^{-z/(2 \lambda_d)} + \half \lambda_d^{-7/2} z^{3/2} e^{-z/(2 \lambda_d)}.
$
		
		We now note that $z \mapsto \lambda_d^{-1}h_{\chi^2}(z/\lambda_d)$ is a density, since it is the density of $\lambda_j Z_j^2$. Also, $z^{1/2} e^{-z/(2 \lambda_d)}$ and $z^{3/2} e^{-z/(2 \lambda_d)}$ are proportional to Gamma-distributions. Recall that the Gamma$(\alpha, \beta)$ density for $\alpha > 0, \beta > 0$ is $h_{G(\alpha, \beta)}(z) = \beta^\alpha  z^{\alpha-1} e^{-\beta x}/\Gamma(\alpha) I \{ z \geq 0 \}$ in which $\Gamma(z) = \int_0^\infty u^{z-1} e^{-u} \, \d u$. In conclusion, we see that
$
			\frac{\partial}{\partial \lambda_d} h_d(z)  = - \lambda_d^{-2} h_{\chi^2}(z/\lambda_d) - \half \lambda_d^{-5/2} \frac{\Gamma(3/2)}{(2 \lambda_d)^{3/2}} h_{G(3/2, 1/(2\lambda_d))}(z) 
			+ \half \lambda_d^{-7/2}  \frac{\Gamma(5/2)}{(2 \lambda_d)^{5/2}} h_{G(5/2, 1/(2\lambda_d))}(z) 
			 = - \lambda_d^{-2} h_{\chi^2}(z/\lambda_d) - 2^{-5/2} \lambda_d^{-4} \Gamma(3/2) h_{G(3/2, 1/(2\lambda_d))}(z) 
			+ 2^{-7/2} \lambda_d^{-6}  \Gamma(5/2) h_{G(5/2, 1/(2\lambda_d))}(z)
$
		
By the linearity of integration and $x_{d-1} \mapsto h_{\chi^2}(x_{d-1}/\lambda_d)/\lambda_d$, and $x_{d-1} \mapsto h_{G(5/2, 1/(2\lambda_d))}(x_{d-1})$, and $x_{d-1} \mapsto h_{G(5/2, 1/(2\lambda_d))}(x_{d-1})$ are densities, eq.~\eqref{equ::convolutionDeriv} is a weighted sum of convolutions of densities that result in new densities $h_A, h_B$ and $h_C$. That is, 
$\frac{\partial}{\partial \lambda_d} h_S(s) = - \lambda_d^{-1} h_A(z) - 2^{-5/2} \lambda_d^{-4} \Gamma(3/2) h_B(z) + 2^{-7/2} \lambda_d^{-6}  \Gamma(5/2) f_C(z)$.		
		Returning to eq.~\eqref{equ::FderivlambdaDef} we therefore see that
$
			\frac{\partial}{\partial \lambda_d} H_S(q)  =  \int_0^q - \lambda_d^{-1} h_A(s) - 2^{-5/2} \lambda_d^{-4} \Gamma(3/2) h_B(s) + 2^{-7/2} \lambda_d^{-6}  \Gamma(5/2) h_C(s) \, \d s 
			 = - \lambda_d^{-1} H_A(q) - \half \lambda_d^{-4} \Gamma(3/2) H_B(q) + 2^{-7/2} \lambda_d^{-6}  \Gamma(5/2) H_C(q)
$
		where $H_A, H_B, H_C$ are the cumulative distribution functions of $h_A, h_B, h_C$. 
		
		Recalling that cumulative distribution functions are probabilities, and hence has absolute values bounded by $1$, we see that
$
			|H_j(T_n; x + r_n h_n)| \leq |\lambda_j + r_n (\hat \lambda_j - \lambda_j)|^{-1} + 2^{-5/2} |\lambda_j + r_n (\hat \lambda_j - \lambda_j)|^{-4} \Gamma(3/2) 
			+ 2^{-7/2} |\lambda_j + r_n (\hat \lambda_j - \lambda_j)|^{-6} \Gamma(5/2).
$
		Since $0 \leq r_n \leq 1$ and $\hat \lambda_j \npconv \lambda_j > 0$, we see that $|H_j(T_n; x + r_n h_n)| = O_P(1)$. 
	\end{proof}

\end{appendix}

\end{document}